\documentclass[10pt,leqno]{amsart}
\usepackage{graphicx}
\usepackage{indentfirst,csquotes}

\usepackage[english]{babel}
\usepackage{dsfont}
\newcommand{\one}{\mathds{1}}

\usepackage{fancyhdr} 
\usepackage{subfig}

\usepackage{hyperref}
\usepackage{tikz}
\usetikzlibrary{calc}
\usepackage{amssymb}
\usepackage{mathtools}
\usepackage{amsmath}
\usepackage{latexsym}
\usepackage{amsthm}
\usepackage{esint}

\usepackage{comment}

\usepackage{etoolbox}
\usepackage{dirtytalk}
\usepackage{comment}
\usepackage{hyperref}
\hypersetup{
	colorlinks=true,
	linkcolor=blue,
	citecolor=teal
}

\newcommand{\R}[0]{\mathbb{R}}

\newcommand{\p}{\partial}
\newcommand{\dd}{\, \mathrm{d}}

\newcommand{\Hau}{\mathcal{H}}

\newcommand{\dist}{\mathrm{dist}}

\newcommand{\delomega}{\partial\Omega}

\newcommand{\dH}{\, \mathrm{d} \mathcal{H}^{n-1}}

\newcommand{\argmin}{\operatorname{argmin}}

\newtheorem{Theorem}{Theorem}
\numberwithin{Theorem}{section}

\newtheorem{Lemma}[Theorem]{Lemma}

\theoremstyle{definition}
\newtheorem{Definition}[Theorem]{Definition}
\theoremstyle{remark}

\numberwithin{equation}{section}
\numberwithin{equation}{section}
\DeclareMathOperator*{\Span}{span}

\makeatletter
\newenvironment{customproof}[1]{%
	\par\pushQED{\qed}%
	\normalfont\topsep6pt \trivlist
	\item[\hskip\labelsep\itshape
	Proof of #1\@addpunct{.}]\ignorespaces
}{%
	\popQED\endtrivlist\@endpefalse
}
\makeatother

\begin{document}
	\title[Continuity and monotonicity of weighted perimeters of convex bodies]{Continuity and monotonicity of weighted perimeters of convex bodies} 
	\author[]{Gyula Csat\'o and Davide Giovagnoli}
	\date{\today}
	\address{Gyula Csató: Facultat de Matemàtiques i Informàtica, Universitat de Barcelona\\ 
		Gran Via de les Corts Catalanes, 585, 08007, Barcelona, Spain
	}
	\address{Centre de Recerca Matemàtica, Edifici C, Campus Bellaterra, 08193 Bellaterra, Spain}
	\email{gyula.csato@ub.edu}
	\address{Davide Giovagnoli: Dipartimento di Matematica, Università di Bologna, Piazza di Porta \\ S.Donato 5, 40126, Bologna-Italy}
	\email{d.giovagnoli@unibo.it}
	
	\maketitle
	\begingroup
	\let\thefootnote\relax
	\footnotetext{\small  Funding: 
     Part of this project was done during a visit of D.G. at the Universitat de Barcelona and during a visit of G.C. at the University of Bologna. The kind hospitality of both departments is acknowledged.
     G.C. was supported by the Spanish AEIs projects PID2024-156429NB-I00, PID2021-125021NA-I00, and PID2021-123903NB-I00. This work was supported by Gruppo Nazionale per l'Analisi Matematica la
Probabilità e le loro Applicazioni - INdAM.
} 
\endgroup
		\begin{abstract}We investigate weighted perimeters of convex bodies in $\mathbb{R}^n$.
    Our main result states that the weighted perimeter  is continuous with respect to the Hausdorff metric for every radial weight function. We also establish monotonicity under inclusion under suitable assumptions on the radial weight, and show that these assumptions are sharp. These results are applied to show the existence of minimizers for a shape optimization problem.       
	\end{abstract}
	
	\tableofcontents

\section{Introduction}

This paper investigates properties of functionals of convex bodies (bounded closed convex sets) $\Omega\subset \R^n$, $n\geq 2,$ defined as follows. Given a function $\phi:(0,\infty)\to\R$ we set 
$$
  E_{\phi}(\delomega):=\eta_{\Omega}\int_{\p\Omega}\phi(|x|)
  \dd \Hau^{n-1}(x)\quad\text{where}\quad
  \eta_{\Omega}= \left\{\begin{array}{rl}
    1 & \text{ if $\Omega$ has a nonempty interior,}
    \\
    2 & \text{ if $\Omega$ degenerate.}
  \end{array}\right.
$$
Here, a degenerate convex body $\Omega$ is defined as a convex body which has empty interior with respect to the standard topology in $\R^n$. For a convex body, degeneracy is equivalent to being contained in an $(n-1)$-dimensional hyperplane. It is also equivalent to $\p\Omega=\Omega$, where $\p\Omega$ denotes the topological boundary of $\Omega$. $\mathcal{H}^{n-1}$ denotes the $(n-1)$-dimensional Hausdorff measure. 

The main example we have in mind is the weighted perimeter defined by
\begin{equation}\label{intro:def:Ep}
  E_p(\delomega):=\eta_{\Omega}\int_{\p\Omega}|x|^p
  \dd \Hau^{n-1}(x)\quad\text{ where }p>-(n-1).
\end{equation}
Let us  mention that $E_2$ is the angular moment of inertia of $\delomega$ which has been studied for instance in \cite{Hurwitz1902,SachsIandII,Hall1985,FreitasLaugesenLiddell,Charlotte2024,LaMannaSannipoli}.
In \cite{CGR} we have introduced some new optimization problems for the $p$-th moments $E_p$ under a standard perimeter constraint.  


The principal goal of this note is to generalize the known result on the continuity of the standard perimeter of convex bodies with respect to the Hausdorff metric. This main result of the present paper, Theorem \ref{thm:continuity_energy} below, has been announced and used in \cite{CGR} in a crucial way. We present a self-contained proof here because the result is of independent interest.

Proofs of the continuity of the standard perimeter can be found in \cite[Result 3.2.36, p. 272]{Federer} or in \cite[Section 2.3]{Bucur}. In the case where the weighted perimeter is $E_2$ and $n=3$ the continuity and monotonicity have already been discussed in \cite{FreitasLaugesenLiddell}, but we  are not aware of a proof in the literature. Moreover, we will deal with more general radial weights than just powers of the distance function, and in any dimension. 

With our method we also deal with the monotonicity of weighted perimeters of convex sets. 
Monotonicity can only hold under some additional assumptions, due to the result in \cite[Theorem 1.1]{saracco2024monotonicity}, where it has been proven that if  $E_\phi(\p A)\leq E_\phi(\p B)$ for every pair of convex bodies $A\subset B$, then $\phi$ must be constant. Our Theorem \ref{thm:monotonicity} provides such appropriate assumptions. Finally, we will apply this theorem to give a simple proof of the existence of minimizers for a certain minimization problem among convex bodies.  

Before stating the main results, we recall the Hausdorff distance and introduce the admissible class of weights. Firstly,
the notion of Hausdorff metric $d_H$ (also usually called the Hausdorff distance in the literature). 
For $\Omega_1,\Omega_2 \subset \R^n$ this distance is defined as
\[
d_H(\Omega_1,\Omega_2):= \inf \{\delta >0 \, : \, \Omega_1 \subset \Omega_2 + B_\delta, \, \Omega_2 \subset \Omega_1 + B_\delta \},
\]
where we use the notation of Minkowski sum of two sets. The importance of this metric, especially in our applications, stems from the Blaschke selection principle, which states that uniformly bounded families of convex bodies are precompact in this metric.  

Secondly, let us introduce the class of weight functions we will deal with. As a motivation consider the weights $\phi(t)=t^p$ appearing in \eqref{intro:def:Ep}. The condition $p>-(n-1)$ is equivalent to the fact that $|x|^p$ is locally integrable on $(n-1)$-dimensional hypersurfaces, or in other words
$x' \in \R^{n-1} \mapsto \phi(|x'|) $ is integrable near the origin in $\R^{n-1}$, that is
\begin{equation}\label{eq:integrability_phi}
\int_0^{1} |\phi(t)| t^{n-2} \dd t < \infty.
\end{equation}
It is easy to see (see Section \ref{Sec:continuity}) that \eqref{eq:integrability_phi} is equivalent to $\int_0^1 |\phi'(t)| t^{n-1} \dd t<\infty$ for absolutely continuous functions $\phi$ whose derivative $\phi'$  does not change sign. For $\phi(t)=t^p$ this equivalence is obvious by partial integration.
This motivates the following definition.

\begin{Definition}[Weight function]\label{intro:gy:weight function} A \textit{weight function} $\phi$  satisfies the following assumptions.  $\phi$ is absolutely continuous on $(0,\infty)$ and
\begin{equation} \label{cond_int_derivative}
    \int_0^1 |\phi'(t)| t^{n-1} \dd t<\infty.
\end{equation}
\end{Definition}

We will discuss the relation between \eqref{cond_int_derivative} and \eqref{eq:integrability_phi} in
the beginning of Section \ref{Sec:continuity}. In particular, we will show that \eqref{cond_int_derivative} is stronger and ensures \eqref{eq:integrability_phi}.
The integrability condition \eqref{eq:integrability_phi} guarantees that the energy $E_{\phi}(\delomega)$ is well defined.

We can now introduce the main results of the paper. The first result concerns the continuity of the energy functional $\Omega\mapsto E_\phi(\p\Omega)$.

\begin{Theorem}[Continuity of the energy functional]\label{thm:continuity_energy} Let $n\geq 2$ be an integer, $\phi$ be a weight function according to Definition \ref{intro:gy:weight function} and 
       $(\Omega_i)_{i \geq 1}$ be a sequence of  convex bodies in $\R^n$. If $\Omega_i\to \Omega$ for $i \to \infty$ in the Hausdorff metric then 
       \[
       \lim_{i\to\infty}E_\phi(\p\Omega_i)= E_{\phi}( \p\Omega).
       \]
\end{Theorem}

Our proof of Theorem \ref{thm:continuity_energy} 
is based on a convenient one-dimensional representation formula to express $E_\phi$ obtained through a layer-cake decomposition.
 Using this formula it turns out to be sufficient to study the convergence of the standard perimeters of $\Omega_i$ relative to balls and complements of balls. 

We expect that Theorem \ref{thm:continuity_energy} can also be proven using the methods developed in \cite[Section 3]{RitoreVernadikis}, who proved a kind of equivalence between  convergence in Hausdorff metric and convergence in a so-called Lipschitz distance for convex bodies. But our attempts in this direction seem to suggest that this would lead to a longer and more technical proof. 

The second main result is the following monotonicity property.
 In this paper we make the convention that a function $\phi$ is \textit{increasing} if $\phi(t) \leq \phi(s)$ whenever $t < s$, and we will not employ the perhaps more commonly used expression "nondecreasing". 

\begin{Theorem}[Monotonicity of $E_\phi$ under inclusion] \label{thm:monotonicity}
     Let $n\geq 2$, $\phi$  be a weight function according to Definition \ref{intro:gy:weight function} and assume $\phi$ is increasing and $\phi(0) \ge 0$. If $\Omega_1,$ $\Omega_2$ are two convex bodies such that $0 \in \Omega_1 \subset \Omega_2\subset\R^n$, then $E_\phi(\p \Omega_1) \le E_\phi(\p \Omega_2)$.
\end{Theorem}

Note that $\phi(0)$ is understood as $\lim_{t \to 0^+} \phi(t)$ which is well defined since $\phi$ is increasing.

In the beginning of Section \ref{section:monotonicity} we will exhibit examples to show the optimality of this monotonicity theorem. In particular monotonicity fails if any of the hypotheses of the theorem is removed. 
Firstly, if $\phi$ is a nonnegative decreasing function one cannot expect  $E_\phi(\p \Omega_1) \ge E_\phi(\p \Omega_2)$ to hold in general.
Also the conditions $\phi(0)\geq 0$ and $0\in\Omega_1$ are necessary.

In Section \ref{Sec:existence} we will show that this monotonicity theorem provides a simplification of the proof of existence of minimizers presented in \cite{CGR}, where no monotonicity of $\phi$ was assumed. 

\subsection{Notation.}

Throughout the paper, $B_r$ denotes the $n$-dimensional open ball in $\R^n$ of radius $r$ and center $0$. $\mathcal{H}^{n-1}$ will denote the $(n-1)$-dimensional Hausdorff measure. For a set $A$ we will abbreviate this measure by $|A|:=\mathcal{H}^{n-1}(A)$ and we denote $\beta_{n-1}=|\p B_1|.$ 

The standard perimeter of a convex body $\Omega$ with multiplicity will be denoted by $E_0(\p\Omega)$ in accordance with the definition of $E_p$ in \eqref{intro:def:Ep}. 
Specifically, $E_0(\p \Omega)$ is equal to $|\p \Omega|$ if $\Omega$ is nondegenerate, and to $2|\p \Omega|$ if $\Omega$ is degenerate.

In addition, it is convenient to introduce the following notation for the relative perimeter of a convex body $\Omega$.
	For any set $A \subset \R^n$ we define
	\begin{equation} \label{eq:notation_relative_energy}
		E_0(\p \Omega;A):= \begin{cases}
			|\p \Omega \cap A| &\text{if $\Omega$ is nondegenerate}, \\
			2|\p \Omega \cap A| &\text{if $\Omega$ is degenerate}.
		\end{cases}
	\end{equation}
Finally, $A^c$ denotes the complement $\R^n\setminus A$ of a set $A.$

\section{Continuity of the weighted perimeter} \label{Sec:continuity}

The aim of this section is to prove Theorem \ref{thm:continuity_energy}. Before doing so, let us show that indeed the condition \eqref{cond_int_derivative} on the weight function ensures the local integrability  condition \eqref{eq:integrability_phi} on the weight.  
Since $\phi$ is absolutely continuous on $(0,\infty)$, for any $t \in (0,1]$ we can bound $|\phi(t)| \le |\phi(1)| + \int_t^1 |\phi'(s)| \dd s$. Multiplying by $t^{n-2}$, integrating over $(0,1)$, and applying Tonelli's theorem to swap the order of integration yields
\begin{align*}
    \int_0^1 |\phi(t)| t^{n-2} \dd t &\le \int_0^1 |\phi(1)| t^{n-2} \dd t + \int_0^1 |\phi'(s)| \left( \int_0^s t^{n-2} \dd t \right) \dd s \\
    &= \frac{|\phi(1)|}{n-1} + \frac{1}{n-1} \int_0^1 |\phi'(s)| s^{n-1} \dd s,
\end{align*}
which is finite thanks to \eqref{cond_int_derivative}. 

Conversely, let us assume that $\phi'$ has a constant sign almost everywhere on $(0, 1)$. We can show that the local integrability condition \eqref{eq:integrability_phi} implies \eqref{cond_int_derivative}.
Using Tonelli's theorem as above, we can rewrite the integral of the derivative as
\begin{align*} 
    \int_0^1 \lvert \phi'(s) \rvert s^{n-1} \dd s 
    &= (n-1) \int_0^1 t^{n-2} \left( \int_t^1 \lvert \phi'(s) \rvert \dd s \right) \dd t. 
\end{align*}
If $\phi' \ge 0$ almost everywhere on $(0, 1)$, we have $\int_t^1 \lvert \phi'(s) \rvert \dd s = \phi(1) - \phi(t)$. Substituting this into the previous identity yields
\begin{align*} 
    \int_0^1 \lvert \phi'(s) \rvert s^{n-1} \dd s &= (n-1) \int_0^1 t^{n-2} (\phi(1) - \phi(t)) \dd t = \phi(1) - (n-1) \int_0^1 \phi(t) t^{n-2} \dd t. 
\end{align*}
Thanks to \eqref{eq:integrability_phi}, the right-hand side is finite. If $\phi'\leq 0$ the previous identity has to be multiplied by $-1.$ 

Let us also remark that \eqref{eq:integrability_phi} and \eqref{cond_int_derivative} are not equivalent if $\phi'$ changes sign. An example of a weight that satisfies \eqref{eq:integrability_phi} but not \eqref{cond_int_derivative} is $\phi(t)=\cos(\frac{1}{t^2})$ in the case $n=3.$

We begin by expressing conveniently $E_\phi(\p \Omega)$ via a layer-cake decomposition.

\begin{Lemma}\label{lm:rep_formula} Assume $\phi$ is a weight function according to Definition \ref{intro:gy:weight function}.
Let $\Omega$ be a convex body in $\R^n$. Then for every $R>0$ such that $\Omega\subset B_R$ the following formula holds
\begin{equation} \label{eq:layer_cake_formula}
    E_\phi(\p \Omega) = \phi(R) E_0(\p \Omega) - \int_0^R \phi'(t) E_0(\p \Omega; B_t) \dd t.
\end{equation}
If moreover $\phi' \in L^1([0,1])$, the following formula holds
\begin{align} \label{eq:formula_increasing}
   E_{\phi}(\p\Omega)&=\phi(0) \, E_0(\p \Omega) + \int_0^\infty \phi'(t) E_0(\p \Omega; B_t^c) \dd t.
\end{align}
In particular \eqref{eq:formula_increasing} holds true if $\phi$ is monotone and $\phi(0)=\lim_{t\searrow 0}\phi(t) \in \R$.
\end{Lemma} 

As an example, if  $\phi(t)=t^p$ with $p>0$, \eqref{eq:formula_increasing} takes the form $
	 E_{p}(\p\Omega) = \int_0^\infty pt^{p-1} E_0( \p \Omega; B_t^c) \dd t.$
\smallskip

\begin{proof}
    By means of the fundamental theorem of calculus, we express $\phi(|x|)$ for a.e. $x \in \p \Omega$ as 
    \[
    \phi(|x|) = \phi(R) - \int_{|x|}^R \phi'(t) \dd t = \phi(R) - \int_0^R \phi'(t) \one_{(|x|,\infty)}(t) \dd t\quad\text{ if }|x|<R.
    \]
    Let us assume for now that $\Omega$ is a nondegenerate convex body. Plugging the expression above in $E_\phi(\p \Omega)$ yields
    \begin{equation} \label{eq:intermediate}
    \begin{aligned}
        E_\phi(\p \Omega) &= \phi(R)E_0(\p \Omega) - \int_{\p \Omega}  \int_0^R \phi'(t) \one_{(|x|,\infty)}(t)\dd t \dH(x).
    \end{aligned}
    \end{equation}
    To apply Fubini-Tonelli theorem we check the integrability condition. By monotonicity of the perimeter under inclusion for convex bodies (see also \eqref{eq:monoton incl with multipl}) we know that $E_0(\p \Omega; B_t) \le |\p B_t|=\beta_{n-1}t^{n-1}$ for all $t>0$. Thus,
    \begin{equation} \label{necessity_cond}
    \begin{aligned}
        \int_0^R |\phi'(t)| \left( \int_{\p \Omega} \one_{(|x|,\infty)}(t)\dH(x) \right) \dd t &= \int_0^R |\phi'(t)| E_0(\p \Omega; B_t) \dd t \\ & \le \beta_{n-1} \int_0^R |\phi'(t)| t^{n-1} \dd t < \infty.
    \end{aligned}
    \end{equation}
    Notice that the last term on the right-hand side above is finite by \eqref{cond_int_derivative}. Therefore, exchanging the order of integration, and recalling the notation \eqref{eq:notation_relative_energy}, we obtain
    \begin{equation*}
        E_\phi(\p \Omega) = \phi(R) E_0(\p \Omega) - \int_0^R \phi'(t) E_0(\p \Omega; B_t) \dd t.
    \end{equation*}
    
    If $\Omega$ is degenerate, a factor $2$ has to be added in front of both the second term on the right-hand side in \eqref{eq:intermediate} and  in the left-hand side of \eqref{necessity_cond}, and the conclusion \eqref{eq:layer_cake_formula} remains unchanged.

    To prove the remaining formulas, since $R>0$ is such that $\Omega \subset B_R$,  for all $t \ge R$, we have $E_0(\p \Omega; B_t) = E_0(\p \Omega)$ and $E_0(\p \Omega; B_t^c) = 0$.

  Using the identity $E_0(\p \Omega; B_t) = E_0(\p \Omega) - E_0(\p \Omega; B_t^c)$ 
    in \eqref{eq:layer_cake_formula} and since $\phi' \in L^1([0,1])$, we get
    \begin{align*}
        E_\phi(\p \Omega) &= \phi(R) E_0(\p \Omega) - \int_0^R \phi'(t) \left( E_0(\p \Omega) - E_0(\p \Omega; B_t^c) \right) \dd t \\
        &= \phi(R) E_0(\p \Omega) - E_0(\p \Omega) \int_0^R \phi'(t) \dd t + \int_0^R \phi'(t) E_0(\p \Omega; B_t^c) \dd t.
    \end{align*}
   By the fundamental theorem of calculus $\int_0^R \phi'(t) \dd t = \phi(R) - \phi(0)$.  Substituting this back, the first two terms simplify to $\phi(0)E_0(\p\Omega)$. Furthermore, since $E_0(\p \Omega; B_t^c) = 0$ for all $t \ge R$, we can extend the integration domain of the last integral to $[0, \infty)$, yielding \eqref{eq:formula_increasing}.

\end{proof}

Lemma \ref{lm:rep_formula} exhibits the connection between the continuity of $E_\phi$ and that of the standard perimeter, respectively that of  the standard perimeter  relative to a ball.

	In the next lemma we recall the continuity of the standard perimeter functional for convex bodies with respect to the Hausdorff metric. This property of the perimeter functional can be found for instance  in \cite[Result 3.2.36 page 272]{Federer} or in \cite{Bucur}. The proof below is essentially the same as the one of \cite{Federer}. We repeat it here for completeness and to emphasize that it is here where the factor $2$ in the definition of $E_0$ for degenerate sets enters the picture. Moreover, some  ingredients of the proof, namely \eqref{eq:gy:cauchy surface general} and \eqref{eq:monoton incl with multipl}, will be required later. 
    

	\begin{Lemma}[Continuity of the perimeter functional] \label{lm:continuity_perimeter}
		Let $(\Omega_i)_{i \geq 1}$ be a  sequence of convex bodies. Assume that $\Omega_i$ converges to  $\Omega$ in the Hausdorff metric as $i \to \infty$. Then 
		\[
		\lim_{i \to \infty} E_0(\p \Omega_i) = E_{0}(\p \Omega).
		\]
	\end{Lemma}

	Given $\sigma$ in the unit sphere $ \mathbb{S}^{n-1}=\p B_1$, we denote by $\pi_\sigma(K)$ the projection of $K$ onto the hyperplane $H_\sigma$ orthogonal to $\sigma$. For a bounded convex body $K$ with nonempty interior we recall the Cauchy surface area formula (see for instance \cite[Lemma A2]{cabre2022fractional} for a short proof for convex polytopes) which states that
	\begin{equation}\label{eq:gy:cauchy surface for open}
		|\partial K| = \frac{1}{\omega_{n-1}}\int_{\mathbb{S}^{n-1}} |\pi_{\sigma}(K)| \dH(\sigma),
	\end{equation}
where $\omega_{n-1}$ denotes the $(n-1)$-dimensional Lebesgue measure of the unit ball in $\R^{n-1}.$
This formula easily extends to degenerate convex bodies after replacing $|\partial K|$ by $E_0(\p K)=2|\p K|$ on the left hand side as we shall now explain. 
    Indeed if $K$ is degenerate, then after rotation we can assume that $K=K'\times\{0\}$ for some convex body $K'\subset\R^{n-1}.$ Then apply \eqref{eq:gy:cauchy surface for open} to the sets
	$$
	K_h:=\{(x',x_n)\in\R^n;\, x'\in K',\quad x_n\in [-h,h]\},\qquad h>0
	$$
	and take the limit $h\to 0.$ Clearly, on the left hand side of \eqref{eq:gy:cauchy surface for open} we obtain that $|\partial K_h|\to E_0(\p K)$ and on the right hand side we can use dominated convergence theorem and the fact that
	$$
	|\pi_{\sigma}(K)|\leq |\pi_{\sigma}(K_h)|\leq |\pi_{\sigma}(K)|+2h\mathcal{H}^{n-2}(\p K'),
	$$
	to obtain
	\begin{equation}\label{eq:gy:cauchy surface general}
		E_0( \p K) = \frac{1}{\omega_{n-1}}\int_{\mathbb{S}^{n-1}} |\pi_{\sigma}(K)| \dH(\sigma)\quad\text{ for every convex body }K\subset\R^n.
	\end{equation}
The identity \eqref{eq:gy:cauchy surface general} yields the monotonicity of the standard perimeter with respect to inclusion in the following way. If $K_1\subset K_2$ then $\pi_{\sigma}(K_1)\subset\pi_{\sigma}(K_2)$ for all $\sigma$ and thus
\begin{equation}
\label{eq:monoton incl with multipl}
  E_0(\p K_1)\leq E_0(\p K_2).
\end{equation}
Note that \eqref{eq:monoton incl with multipl} is more precise than the usual statement in the literature $|\p K_1|\leq |\p K_2|$. 
	
	We are now ready to provide the proof of Lemma \ref{lm:continuity_perimeter}.

\begin{customproof}{Lemma \ref{lm:continuity_perimeter}}
	
		We claim that for every $\sigma \in \mathbb{S}^{n-1}$, we have
		\begin{equation} \label{eq:claim_peri}
			|\pi_{\sigma}(\Omega_i)| \to |\pi_{\sigma}(\Omega)| \quad \text{as } i \to \infty.
		\end{equation}
		Assuming this claim holds true, we can  readily conclude the proof by applying the dominated convergence theorem to the formula \eqref{eq:gy:cauchy surface general}.
		Indeed, since the sequence $(\Omega_i)_{i \geq 1}$ is bounded, there exists $R > 0$ such that $\Omega_i \subset B_R$ for all $i \geq 1$. Consequently, for any direction $\sigma \in \mathbb{S}^{n-1}$, the projection $\pi_\sigma(\Omega_i)$ is contained in the projection of the $n$-dimensional ball $B_R$, which is an $(n-1)$-dimensional ball of radius $R$. This provides the uniform bound
		\[
		|\pi_{\sigma}(\Omega_i)| \leq |\pi_\sigma(B_R)| = \omega_{n-1} R^{n-1}.
		\]
		Dominated convergence theorem allows us to pass the limit inside the integral, yielding
		\[
		\lim_{i \to \infty} E_0(\partial \Omega_i) = \lim_{i \to \infty} \frac{1}{\omega_{n-1}} \int_{\mathbb{S}^{n-1}} |\pi_{\sigma}(\Omega_i)| \dH(\sigma) = \frac{1}{\omega_{n-1}} \int_{\mathbb{S}^{n-1}} |\pi_{\sigma}(\Omega)| \dH(\sigma) = E_0(\partial \Omega).
		\]
		
We now turn to the proof of \eqref{eq:claim_peri}.
In view of the assumed Hausdorff convergence of $\Omega_i$, for every $\varepsilon > 0$ there exists an $i_0$  such that
		\[
		\Omega_i \subseteq \Omega +  B_\varepsilon \quad \text{and} \quad \Omega \subseteq \Omega_i +  B_\varepsilon \qquad \text{for any }i\geq i_0.
		\]
		Exploiting the linearity of the projection, we apply $\pi_{\sigma}$ to both inclusions
		\[
		\pi_{\sigma}(\Omega_i) \subseteq \pi_{\sigma}(\Omega) +  \pi_{\sigma}(B_\varepsilon) \quad \text{and} \quad \pi_{\sigma}(\Omega) \subseteq \pi_{\sigma}(\Omega_i) + \pi_{\sigma}(B_\varepsilon)
		\]
		Since $\pi_{\sigma}(B_1) = B_1^{n-1}$ (the $(n-1)$-dimensional unit ball in the hyperplane $H_\sigma$), we obtain
		\[
		\pi_{\sigma}(\Omega_i) \subseteq \pi_{\sigma}(\Omega) +  B_\varepsilon^{n-1} \quad \text{and} \quad \pi_{\sigma}(\Omega) \subseteq \pi_{\sigma}(\Omega_i) +  B_\varepsilon^{n-1}.
		\]
		This implies that the Hausdorff distance between the projections is bounded by that of the original sets
		\[
		d_H(\pi_{\sigma}(\Omega), \pi_{\sigma}(\Omega_i)) \le d_H(\Omega, \Omega_i) < \varepsilon, \qquad \text{for every }i\geq i_0.
		\]
		Thus $\pi_{\sigma}(\Omega_i) \to \pi_{\sigma}(\Omega)$ in the Hausdorff distance. Since these projections are convex bodies in $\R^{n-1}$, and the $(n-1)$-dimensional Lebesgue measure is continuous with respect to the Hausdorff metric on the class of convex bodies in $\R^{n-1}$ (see, for instance, \cite[Theorem 1.8.20]{Schneider}), it follows that $|\pi_{\sigma}(\Omega_i)| \to |\pi_{\sigma}(\Omega)|$.
	\end{customproof}

We now prove an extension of the Cauchy surface area formula \eqref{eq:gy:cauchy surface general} with multiplicity to the perimeter of a convex body relative to a ball. Recall our notation of the perimeter relative to a set given in \eqref{eq:notation_relative_energy}.

\begin{Lemma}
\label{Lemma:Cauhy formula relative a ball}
Let $\Omega\subset\R^n$ be a convex body. Then 
$$
  E_0(\delomega;B_t)=\frac{1}{\omega_{n-1}}\int_{\mathbb{S}^{n-1}}|\pi_{\sigma}(\Omega\cap\overline{B}_t)|\dH(\sigma)-|\Omega\cap\partial B_t|\quad\text{ for every $t>0$}.
$$
\end{Lemma}

\begin{proof}
If $\Omega\cap\overline{B_t}$ is empty or consists of a single point, the conclusion is immediate. We may therefore assume otherwise.

We apply \eqref{eq:gy:cauchy surface general} to the convex body $\Omega\cap\overline{B_t}$, obtaining
$$
  E_0(\p(\Omega\cap\overline{B}_t))=\frac{1}{\omega_{n-1}}\int_{\mathbb{S}^{n-1}}|\pi_{\sigma}(\Omega\cap\overline{B}_t)|\dH(\sigma).
$$
Thus to prove the lemma it is sufficient to show that
\begin{equation}
  \label{eq:E0 splitting ae t Bt}
   E_0(\p(\Omega\cap\overline{B}_t))= E_0(\delomega;B_t)+|\Omega\cap\partial B_t|\quad\text{ for $t>0$.}
\end{equation}
We now distinguish two cases, depending on the degeneracy of $\Omega.$

If $\Omega$ is degenerate, then also $\Omega\cap \overline{B}_t$ is degenerate. Moreover, note that $|\Omega\cap \p B_t|=0$, since  a sphere intersected with any hyperplane has vanishing $\mathcal{H}^{n-1}$-measure. Using this fact twice, and that $\delomega=\Omega$, we obtain 
\begin{align*}
  E_0(\p(\Omega\cap\overline{B}_t))=&2|\Omega\cap \overline{B}_t|=2|\Omega\cap B_t|
  =2|\delomega\cap B_t|
  \\
  =&E_0(\delomega;B_t)=E_0(\delomega;B_t)+|\Omega\cap \p B_t|.
\end{align*}
This proves \eqref{eq:E0 splitting ae t Bt} in the present case. 

Assume now that $\Omega$ is nondegenerate. Firstly, recall that for any two closed sets $C,D \subset \R^n$ we have
\begin{equation*} \label{eq:boundary_intersection}
    \p(C \cap D) =( \p C \cap D) \cup (C \cap \p D).
\end{equation*}
Using this identity and that $\p\Omega\cap \p B_t\subset \Omega\cap \p B_t$ gives that the boundary of the convex body $\Omega\cap\overline{B}_t$ is given by 
\[
\partial(\Omega\cap\overline{B}_t) = \bigl(\partial\Omega\cap \overline{B}_t\bigr) \cup \bigl(\Omega\cap\partial B_t\bigr)= \bigl(\partial\Omega\cap B_t\bigr) \cup \bigl(\Omega\cap\partial B_t\bigr).  
\] 
Note that the last expression on the right is a disjoint union.

Observe now that $B_t$ intersected with the interior of $\Omega$ cannot be empty, because if it were empty, then a hyperplane would separate $\Omega$ from $B_t$. And in such a case the intersection $\Omega\cap \overline{B}_t$ would be at most one point, which we have excluded already. Hence we have that $E_0(\partial(\Omega\cap\overline{B}_t)) = |\partial(\Omega\cap\overline{B}_t)|$ and, 
consequently,
\[ 
\begin{aligned} E_0\bigl(\partial(\Omega\cap\overline{B}_t)\bigr) &= \left|\partial(\Omega\cap\overline{B}_t)\right|= \left|\partial\Omega\cap B_t\right| + \left|\Omega\cap\partial B_t\right| = E_0(\partial\Omega;B_t) + \left|\Omega\cap\partial B_t\right|,\end{aligned} 
\] 
which proves \eqref{eq:E0 splitting ae t Bt} in the second case.

\end{proof}

We now recall a fact from geometric measure theory. Namely that for every convex body $D \subset \R^n$ we have
		\begin{equation} \label{eq:measure_zero_pBt}
			|\p D\cap \p B_t| = 0 \quad \text{ for a.e. }t>0.
		\end{equation}
		One can show this by applying the coarea formula on rectifiable sets (see for instance \cite[Theorem II.7.8]{maggi2012sets}) to the $\mathcal{H}^{n-1}$-rectifiable set $\p D$ with the Lipschitz function $u(x)=|x|$ which gives
		\begin{equation} \label{coarea}
			\int_{0}^\infty \mathcal{H}^{n-2} \left( \p D \cap \p B_t \right) \dd t  = \int_{\p D} \left|\nabla_T |x|\right| \dH(x), 
		\end{equation}
		where $\nabla_T$ denotes the tangential gradient on $\p D$. In particular since $\left|\nabla_T |x|\right| \leq 1$ for a.e. $x \in \p D$, \eqref{coarea} implies that the integral on the left-hand side is bounded, and thus
		\[
		\mathcal{H}^{n-2} \left( \p D \cap \p B_t \right) <\infty  \quad \text{ for a.e. }t>0.
		\]
		Therefore, for almost every $t>0$ we have $\mathcal{H}^{n-1}(\p D\cap \p B_t)=0,$ which is precisely \eqref{eq:measure_zero_pBt}.
    
We now establish the aforementioned continuity of the perimeter relative to a ball for convex bodies. 

	\begin{Lemma}[Continuity of the perimeter functional relative to a ball] \label{lm:continuity_relative_perimeter}
		Let $(\Omega_i)_{i \geq 1}$ be a sequence of convex bodies. Assume that $\Omega_i$ converges to $ \Omega$ in the Hausdorff metric for some convex body $\Omega$ as $i \to \infty$ . Then for a.e. $t>0$ it holds
		\begin{equation*}
			\begin{aligned}
				&  \lim_{i \to \infty} E_0(\p \Omega_i; B_t)=  E_0(\p \Omega; B_t) .
			\end{aligned}
		\end{equation*}
	\end{Lemma}

\begin{proof}
Since the formula of Lemma \ref{Lemma:Cauhy formula relative a ball} holds for all $\Omega_i$ and for $\Omega$, by dominated convergence theorem, as in Lemma \ref{lm:continuity_perimeter}, it is sufficient to prove for a given $t>0$, that for every $\sigma\in \mathbb{S}^{n-1}$
\begin{equation}
    \label{eq:two limits}
    \lim_{i\to\infty}|\pi_{\sigma}(\Omega_i\cap \overline{B}_t)=|\pi_{\sigma}(\Omega\cap \overline{B}_t)|\quad \text{ and }\quad \lim_{i\to\infty}|\Omega_i\cap \p B_t|=|\Omega\cap \p B_t|
\end{equation}
The second limit  might hold only for a.e. $t>0$, in case that $\Omega$ is nondegenerate.

Let us first prove the first limit in \eqref{eq:two limits}. We shall distinguish two cases depending on whether $\Omega\cap B_t$ is empty or not.

If $\Omega\cap B_t\neq \emptyset$ then  $\Omega$ and $\overline{B}_t$ cannot be separated by a hyperplane and thus from \cite[Theorem 1.8.10]{Schneider} it follows that $\Omega_i\cap \overline{B}_t$ converges to $\Omega\cap\overline{B}_t$ in the Hausdorff metric. In this case we can proceed exactly as we have proven \eqref{eq:claim_peri} in the proof of Lemma \ref{lm:continuity_perimeter} and obtain the first limit in \eqref{eq:two limits}. 

If $\Omega\cap B_t = \emptyset$, then there exists a separating hyperplane between $\Omega$ and $B_t.$ Without loss of generality $\Omega\subset \R^{n-1}\times(-\infty,-t]$ and $\Omega\cap\overline{B}_t$ contains at most one point, namely $(0,\ldots,0,-t)$. In particular $|\pi_{\sigma}(\Omega\cap \overline{B}_t)|=0$. So in this case it is sufficient to prove that 
\begin{equation}\label{eq:pi sigma Omegi cap}
  \lim_{i\to\infty} |\pi_{\sigma}(\Omega_i\cap \overline{B}_t)|=0.
\end{equation}
By hypothesis $\Omega_i$ converges to $\Omega$ in the Hausdorff metric, and hence for any given $0<\varepsilon<t$ the sets $\Omega_i$ are all contained in $\R^{n-1}\times(-\infty,-t+\varepsilon]$ for $i$ big enough. Thus $\Omega_i\cap \overline{B}_t$ is contained in the cylinder $\{(x',x_n)\in\R^n;\, |x'|\leq r_{\varepsilon}\text{ and }x_n\in [-t,-t+\varepsilon]\}$, where $r_{\varepsilon}^2=t^2-(t-\varepsilon)^2.$ Observe that $r_{\varepsilon}$ converges to zero, as $\varepsilon\to 0,$ and hence \eqref{eq:pi sigma Omegi cap} follows. 

We now deal with the second limit in \eqref{eq:two limits}. For this we distinguish whether $\Omega$ is degenerate or not. If $\Omega$ is degenerate, then $|\Omega\cap \p B_t|=0.$ We may assume, after a rotation, that $\Omega\subset\R^{n-1}\times\{a\}$ for some $a\in\R$. By assumption of the Hausdorff convergence, for every $\varepsilon>0$ there exists an $i_0$ such that
$$
  \Omega_i\subset\R^{n-1}\times[a-\varepsilon,a+\varepsilon]\quad\forall\, i\geq i_0\,.
$$
Thus we obtain that 
$$
  |\Omega_i\cap \p B_t|\leq |\R^{n-1}\times [a-\varepsilon,a+\varepsilon]\cap \p B_t|\to 0
  =|\Omega\cap \p B_t|\quad\text{ for $\varepsilon\to 0$}.
$$
This proves the second limit in \eqref{eq:two limits} in the degenerate case. 

Let us now consider the second limit in \eqref{eq:two limits} when $\Omega$ is not degenerate. In this case we will write the limit in the form
$$
  \lim_{i\to\infty}\int_{\p B_t}\one_{\Omega_i}(x)\dH(x)=\int_{\p B_t}\one_{\Omega}(x)\dH(x)
$$
and will use the dominated convergence theorem to conclude. It is thus sufficient to prove that, for almost every $t>0$
\begin{equation}
    \label{eq:ae t and ae x in pBt}
    \lim_{i\to\infty}\one_{\Omega_i}(x)=\one_{\Omega}(x)\quad\text{ for $\mathcal{H}^{n-1}$-a.e. $x\in \p B_t$}.
\end{equation}

By \eqref{eq:measure_zero_pBt}, we may choose $t>0$ such that 
$|\partial\Omega\cap\partial B_t|=0$.
For such a $t$ it holds that $\mathcal{H}^{n-1}$-a.e. $x\in\p B_t$ belongs to $\R^n\setminus\delomega.$ Therefore we will prove \eqref{eq:ae t and ae x in pBt} precisely for these $x$ distinguishing the two case whether $x$ is in the interior of $\Omega$ or $x\in \Omega^c.$ If $x \in \Omega^c$ then $d:=\dist(x, \Omega)>0$ and for $i$ sufficiently large we have $d_H(\Omega_i,\Omega)<d/2$ which implies $x \notin \Omega_i$ and proves \eqref{eq:ae t and ae x in pBt}.

Consider $x \in \mathring{\Omega}$, where $\mathring{\Omega}$ denotes the interior of $\Omega$ with respect to the standard topology in $\R^n$.
This case will follow easily from \cite[Lemma 1.8.1]{Schneider}, which states that $\Omega_i$ converges to $\Omega$ in the Hausdorff distance if and only if $\p\Omega_i$ converges to $\delomega$ in the same distance. 
If $x\in\mathring{\Omega}$, then $B_{2\varepsilon}(x)\subset\mathring{\Omega}$ for some $\varepsilon>0.$ Using \cite[Lemma 1.8.1]{Schneider} we obtain that $\delomega_i\subset \delomega+B_{\varepsilon}$ for all $i$ big enough. Hence for every $y\in \delomega_i$ there exists $z\in\delomega$ such that $|y-z|<\varepsilon$. This gives that
$$
  |x-y|\geq |x-z|-|z-y|>2\varepsilon-\varepsilon=\varepsilon.
$$
Since $y\in\delomega_i$ was arbitrary, we obtain that $B_{\varepsilon}(x)$ does not intersect $\delomega_i$ and is contained either in the interior of $\Omega_i$ or in $\Omega_i^c.$ The second case is not possible, contradicting the assumption $d_H(\Omega_i,\Omega)\to 0.$ Hence $x\in \Omega_i$ for all $i$ big enough, which proves \eqref{eq:ae t and ae x in pBt} in the case $x\in\mathring{\Omega}.$

\end{proof}

	We have now all the tools to provide the proof of Theorem \ref{thm:continuity_energy}.


\begin{customproof}{Theorem \ref{thm:continuity_energy}}
		Since the sequence $(\Omega_i)_{i \geq 1}$ is  bounded, there exists $R>0$ such that $\Omega_i \subset B_R$ for all $i \geq 1$, which implies $\Omega \subset B_R$ as well.
		Using the representation formula \eqref{eq:layer_cake_formula}, we can write the energy of each $\Omega_i$ as
		\begin{equation}\label{eq:E_phi_representation}
			E_\phi(\p \Omega_i) = \phi(R) E_0(\p \Omega_i) - \int_0^R \phi'(t) E_0(\p \Omega_i; B_t) \dd t.
		\end{equation}
		
		We analyze the convergence of the two terms on the right-hand side individually.
		
		First, by the continuity of the perimeter functional for convex bodies (Lemma \ref{lm:continuity_perimeter}), we immediately have 
		\[
		\lim_{i \to \infty} \phi(R) E_0(\p \Omega_i) = \phi(R) E_0(\p \Omega).
		\]
		For the integral term in \eqref{eq:E_phi_representation}, the integrand is given by $h_i(t) := \phi'(t) E_0(\p \Omega_i; B_t) \one_{[0,R]}(t)$.
		By Lemma \ref{lm:continuity_relative_perimeter}, we know that $E_0(\p \Omega_i; B_t) \to E_0(\p\Omega; B_t)$ for almost every $t > 0$. Thus, $h_i(t)$ converges pointwise almost everywhere to $\phi'(t) E_0(\p \Omega; B_t) \one_{[0,R]}(t)$.
		
		To apply the dominated convergence theorem, we need to find an integrable bounding function. By the monotonicity of the perimeter for convex sets \eqref{eq:monoton incl with multipl}, the relative perimeter of $\Omega_i$ inside the ball $B_t$ is bounded by the perimeter of the ball itself
		\[
		E_0(\p\Omega_i; B_t) \leq |\partial B_t| = \beta_{n-1} t^{n-1}.
		\]
		This yields the uniform pointwise bound
		\[
		|h_i(t)| \leq \beta_{n-1} |\phi'(t)| t^{n-1} \one_{[0,R]}(t).
		\]
	 The assumption \eqref{cond_int_derivative} on the weight function guarantees exactly that the dominating function $|\phi'(t)| t^{n-1}$ is integrable on $[0, R]$. Therefore, by the dominated convergence theorem, we can pass to the limit inside the integral
		\[
		\lim_{i \to \infty} \int_0^R \phi'(t) E_0(\p\Omega_i; B_t) \dd t = \int_0^R \phi'(t) E_0(\p\Omega; B_t)\, \dd t.
		\]
Thus, taking the two limits on the right-hand side of 
\eqref{eq:E_phi_representation}, we obtain $\lim_{i \to \infty} E_\phi(\p \Omega_i) = E_\phi(\p \Omega)$, which concludes the proof.
	\end{customproof}
    
\section{Monotonicity of the weighted perimeter}
\label{section:monotonicity}

Before proving Theorem \ref{thm:monotonicity}, we show that all the assumptions of the theorem are necessary.

Firstly, we emphasize that neither direction of monotonicity holds in general for decreasing weights.
To see this, consider the decreasing weight $\phi(t)=t^p$ with $-(n-1)<p<0$. On one hand, taking a ball $B_r \subset B_1$ with $0<r<1$, a direct computation yields $E_{p}(\p B_{r}) < E_{p}( \p B_{1})$. 
On the other hand, if we consider the degenerate set $B_{1}^{n-1} \times \{ 0\} =\{ (x',0) \in \R^n \, : \, |x'|\leq 1 \}$, which is contained in $B_1$, we have
\[
E_{p}( B_{1}^{n-1} \times \{ 0\} ) > E_{p}( \p B_{1}) \quad \text{for all $p$ sufficiently close to $-(n-1)$} .
\]

Secondly, the assumption $\phi(0)\geq 0$ cannot be removed. Indeed, the weight $\phi(t)=t-1$ is increasing, yet a direct computation reveals that $E_\phi(\p B_{1/4}) > E_\phi(\p B_{1/3})$, whereas $E_\phi(\p B_1) < E_\phi(\p B_2)$.

Finally, the condition $0 \in \Omega_1$ is required too. Consider the increasing weight $\phi(t) = t^2$ in $\R^2$, the unit square $\Omega_2 = [0,1]^2$, and the triangle $\Omega_1 = \{ (x_1,x_2) \in \Omega_2 : x_1+x_2 \ge 1 \}$. We have $\Omega_1 \subset \Omega_2$, but $0 \notin \Omega_1$. Since $\p \Omega_1$ and $\p \Omega_2$ coincide on $\{x_1=1\}$ and $\{x_2=1\}$, we only need to compare the energy on the remaining boundary portions. The segments of $\p \Omega_2$ lying on the axes yield $2 \int_0^1 t^2 \dd t = \frac{2}{3}$, whereas the portion of $\p \Omega_1$ on the diagonal yields $\sqrt{2} \int_0^1 \big(t^2 + (1-t)^2\big) \dd t = \frac{2\sqrt{2}}{3} > \frac{2}{3}$, showing that $E_2(\p \Omega_1) > E_2(\p \Omega_2)$.


We now focus on the proof of Theorem \ref{thm:monotonicity}.

	\begin{customproof}{Theorem \ref{thm:monotonicity}}
Using the monotonicity under inclusion for the perimeter of \eqref{eq:monoton incl with multipl} and the assumption $\phi(0)\geq 0,$ we obtain that $\phi(0)E_0(\p \Omega_1) \le \phi(0)E_0(\p \Omega_2)$.
Hence, due to formula \eqref{eq:formula_increasing}, it remains to prove that 
$\phi'(t)E_0(\p \Omega_1; B_t^c) \le \phi'(t)E_0(\p \Omega_2; B_t^c)$ for all $t\geq 0.$ Since $\phi'\geq 0$ by assumption, it is sufficient to show that
\begin{equation} \label{eq:thesis}
	    E_0(\p \Omega_1; B_t^c) \le E_0(\p \Omega_2; B_t^c) \quad \text{for } t \ge 0.
\end{equation} 
To show this we consider
the metric projection map $\pi:\R^n \to \Omega_1$ defined as
		\begin{equation}\label{eq:argmin}
		\pi(y)=\argmin_{x \in \Omega_1} |y-x|.
		\end{equation}
Let $H_a(x)$ be a separating hyperplane for $\Omega_1$ at $x \in \p \Omega_1$ with outer normal vector $a$: that is $H_a(x)=\{z\in \R^n \, :\, (z-x)\cdot a= 0 \}$ and $\Omega_1 \subset \{z\in \R^n \, :\, (z-x)\cdot a\leq 0 \}$. For any such $a$ the ray $\{ x + t a \, : t \geq 0\}$ must intersect $\p \Omega_2$ at some point $y \in \p \Omega_2$ and $\pi(y)=x$. Hence $ \partial \Omega_1\subset \pi(\partial \Omega_2) $.

Let now $x \in \partial \Omega_1 \cap B_t^c$ and $y \in \Omega_2$ such that $\pi(y)=x$. We claim that $y \in B_t^c$. To show this we use that $y-x$ must be orthogonal to a separating hyperplane for $\Omega_1$. After a rotation, we may assume that $H_{e_n}=\{ z \in \R^n \, : \, z_n=x_n\}$ is a separating hyperplane for $\Omega_1$ such that $\Omega_1 \subset \{ z \leq x_n\}$.
Because $0 \in \Omega_1$, we necessarily have $0 \leq x_n$. 
    Writing the coordinates as $z=(z',z_n) \in \R^{n-1} \times \R$, we have that $y=(x',y_n)$ for some $y_n \geq x_n\geq 0$ and this gives
\[
|y|^2 = |x'|^2 + y_n^2 \geq |x'|^2 + x_n^2 = |x|^2 \geq t^2, 
\]
which proves $y \in B_t^c$. In particular, we have shown that 
\begin{equation}\label{eq:inclusion pi and Btcompl}
 \p \Omega_1 \cap B_t^c \subset \pi \left( \p \Omega_2 \cap B_t^c\right).
\end{equation}

We now distinguish three cases depending on the degeneracy of the sets $\Omega_1,\, \Omega_2$. Firstly, if both $\Omega_1$ and $\Omega_2$ are degenerate then \eqref{eq:thesis} follows immediately from the fact that $\delomega_1=\Omega_1\subset\Omega_2=\delomega_2$.

In the second case let us assume that neither $\Omega_1$ nor $\Omega_2$ is degenerate. In that case we use that $\pi$ is a $1$-Lipschitz map (see for instance \cite[Theorem 1.2.1]{Schneider}) and for any set $A$ it holds (see for instance \cite[Theorem 2.8]{EvansGariepy})
\begin{equation} \label{eq:inclusion_1}
    \mathcal{H}^{n-1}(\pi(A)) \leq   \mathcal{H}^{n-1}(A).
\end{equation}
Thus, using \eqref{eq:inclusion_1} and \eqref{eq:inclusion pi and Btcompl}, we obtain that 
\[
E_0(\p \Omega_1;B_t^c)=\mathcal{H}^{n-1}(\p \Omega_1 \cap B_t^c)\leq\mathcal{H}^{n-1}(\p \Omega_2 \cap B_t^c) = E_0(\p \Omega_2;B_t^c),
\]
establishing \eqref{eq:thesis}.

The third case is when $\Omega_1$ is degenerate and $\Omega_2$ is nondegenerate. We can assume, after a rotation, that $\Omega_1 \subset \{ z \in \R^n \, : \, z_n=0\}$. As $\Omega_2$ is nondegenerate, there must exist some $y^\star \in \Omega_2$ such that $y_n^\star>0$ or $y_n^\star<0$. Let us assume $y_n^\star>0$. Define
\[
\Omega_2^+ = \Omega_2 \cap \{ z_n \geq 0\}, \quad \Omega_2^- = \Omega_2 \cap \{ z_n \leq 0\}.
\]
If $ x \in \p \Omega_1 \cap B_t^c$ then for both $a=\pm e_n$ the plane $H_a(x)$ is a separating hyperplane for $\Omega_1$. Thus, arguing as after \eqref{eq:argmin}, for almost every $ x \in \p \Omega_1 \cap B_t^c$ has at least two preimages under $\pi$, that is, there exist
\[
y^+ \in \p \Omega_2^+ \cap \{ z_n >0\} \quad \text{ and } \quad y^- \in \p \Omega_2^- \cap \{ z_n \leq 0\}
\]
such that $\pi(y^+)=x$ and $\pi(y^-)=x$. We have to now distinguish two subcases: namely $y^-$ can be chosen such that
$$
  \left\{ \begin{array}{rl}
   y^- \in \p \Omega_2^- \cap \{ z_n = 0\} & \text{ if $\Omega_2^-$ is degenerate}
   \\
    y^- \in \p \Omega_2^- \cap \{ z_n < 0\} & \text{ if $\Omega_2^-$ is not degenerate}.
    \end{array}\right.
$$

Therefore, up to an $\mathcal H^{n-1}$-negligible set, it follows that
\[
\p \Omega_1 \cap B_t^c \subset \pi\left( \p \Omega_2^+ \cap \{ z_n >0\} \right),\quad  \p \Omega_1 \cap B_t^c \subset 
\left\{ \begin{array}{l}
  \pi\left( \p \Omega_2^- \cap \{ z_n = 0\} \right) \text{ if $\Omega_2^-$ is degenerate}
   \\
  \pi\left( \p \Omega_2^- \cap \{ z_n < 0\} \right) \text{ if $\Omega_2^-$ is not degenerate}.
    \end{array}\right.
\]
Moreover, we have seen that if $x\in B_t^c,$ then $y^{\pm}\in B_t^c.$ 
Hence, using \eqref{eq:inclusion_1} for the first one of these inclusions, we infer
\begin{equation}\label{eq:Omega2plus}
\mathcal{H}^{n-1}(\p \Omega_1 \cap B_t^c) \leq \mathcal{H}^{n-1}(\p \Omega_2^+ \cap \{ z_n>0\} \cap B_t^c) = \mathcal{H}^{n-1}(\p \Omega_2 \cap \{ z_n>0\} \cap B_t^c). 
\end{equation}
Similarly, depending on whether $\Omega_2^-$ is degenerate or not, we fall in either the first or the second case. If $\Omega_2^-$ is not degenerate, then using  \eqref{eq:inclusion_1}  gives
\begin{equation}\label{eq:Omega2minus}
 \begin{split}
 \mathcal{H}^{n-1}(\p \Omega_1 \cap B_t^c) \leq \mathcal{H}^{n-1}(\p \Omega_2^- \cap \{ z_n < 0\} \cap B_t^c)=& \mathcal{H}^{n-1}(\p \Omega_2 \cap \{ z_n < 0\} \cap B_t^c)
 \\
 =& \mathcal{H}^{n-1}(\p \Omega_2 \cap \{ z_n \leq  0\} \cap B_t^c).
 \end{split}
\end{equation}
In the last identity we have used that $\Omega_2^-$ is convex and nondegenerate.
If $\Omega_2^-$ is degenerate, then $\p\Omega_1=\Omega_1\subset \Omega_2^-=\p\Omega_2^-$ and the final conclusion of \eqref{eq:Omega2minus} is valid as well. 
Hence, summing up \eqref{eq:Omega2plus} and \eqref{eq:Omega2minus} we obtain
\[
E_0(\p \Omega_1;B_t^c)= 2\mathcal{H}^{n-1}(\p \Omega_1 \cap B_t^c) \leq \mathcal{H}^{n-1}(\p \Omega_2 \cap B_t^c)=E_0(\p \Omega_2;B_t^c),
\]
which proves \eqref{eq:thesis}.
The case $\Omega_1$ nondegenerate and $\Omega_2$ degenerate cannot occur due to the assumption $\Omega_1 \subset \Omega_2$, and thus the proof is concluded.

	\end{customproof}

\section{Existence of minimizers}
\label{Sec:existence}

In \cite{CGR} we have established the existence of extremals for general weighted perimeters $E_\phi$ among convex bodies with a standard perimeter constraint. Among others, we have proven that there exists a minimizer  of $E_p$ if $p>0$ among all convex bodies $\Omega$ with $E_0(\p\Omega)=$constant.  However, the proof provided there is considerably more involved than necessary for weights which satisfy the monotonicity assumption of Theorem \ref{thm:monotonicity}, like $E_p$ with $p>0$. We have announced in \cite{CGR} that we shall provide a simpler proof in this case. The relevant step which is simplified is that of achieving the compactness of a minimizing sequence.

In dimension $n=2$, the diameter of a convex set is universally bounded by its perimeter. However, for $n \ge 3$, a sequence of convex bodies can have constant perimeter while the diameters tend to  infinity. To prevent this loss of compactness under the constraint $E_0(\p \Omega)=$ constant, the  next Lemma \ref{lm:bound_on_max_radius_increasing} establishes the crucial bound relating the  energy to the maximal radius of the set. This lemma is the analogue of  \cite[Lemma 2.4]{CGR} whose proof is independent of the monotonicity Theorem \ref{thm:monotonicity} and more involved. Both the statement and the proof of the next lemma are simpler than \cite[Lemma 2.4]{CGR}.

\begin{Lemma} \label{lm:bound_on_max_radius_increasing}

Let $n\geq3$. Let $\phi$  be a weight function according to Definition \ref{intro:gy:weight function} and assume $\phi$ is increasing  with $\phi(0) \ge 0$. Let $\Omega$ be a convex body in $\R^n$ with $\mathcal{H}^{n-1}(\p \Omega) >0$ such that $0 \in {\Omega}$. Define $R$ by $R:=\max \{|x| : x \in \p \Omega\}$. Then, it holds that
    \begin{equation} \label{est:radius_phi_pos}
        \frac{E_{\phi}( \p \Omega)}{E_0(\p \Omega)} \geq \frac{1}{8^{n-1}\,n!}  \, \phi\left(\frac{R}{2}\right). 
    \end{equation}

        
\end{Lemma}

As an example, consider the Lemma \ref{lm:bound_on_max_radius_increasing} applied to the energy $E_p$ with $p> 0$. This provides the estimate
$
E_p(\delomega)\geq C_{n,p} E_0(\delomega)R^p
$
for some constant $C_{n,p}>0$ depending only on $n$ and $p$.

The proof of Lemma \ref{lm:bound_on_max_radius_increasing} is a generalization of  \cite[Lemma 8]{FreitasLaugesenLiddell} where the authors deal with the case of the weight $|x|^2$ in $\R^3$ and our strategy is very similar. Our constant in \eqref{est:radius_phi_pos} is worse in this special case than the one obtained in \cite{FreitasLaugesenLiddell}. For particular weights our estimate can certainly be improved, since the Lemma \ref{lm:est_on_triangle}, a crucial ingredient, is merely a rough estimate of the energy of a simplex.  

The idea is to bound the energy of the convex body $\Omega$ from below by comparing it to the energy of a suitably constructed simplex contained within it and exploit monotonicity. 
Hence, we start by isolating a geometric estimate on the energy of a simplex, presented in the next lemma.

\begin{Lemma}\label{lm:est_on_triangle}
Let $n\geq 2$, let $\phi$ be an increasing weight function, $\phi(0)\geq0$ and $\Delta \subset \mathbb{R}^n$ be an $(n-1)$-dimensional simplex of vertices $0,z_1, z_2, \dots, z_{n-1}$ with $|z_1| \geq |z_2| \geq \dots \geq |z_{n-1}|$. Then the following estimate holds
\begin{equation}
\frac{E_\phi(\Delta)}{E_0(\Delta)} \ge \frac{1}{4^{n-1}} \, \phi\left(\frac{|z_1|}{2}\right).
\end{equation}

\end{Lemma}

\begin{proof}
The simplex $\Delta$ is a degenerate convex body and can be parameterized by the linear map $\varphi(t_1,\dots,t_{n-1})$ as
\[
\varphi:(t_1,\dots,t_{n-1})\in T \mapsto \sum_{i=1}^{n-1} t_i \, z_i \quad \text{with } T := \{(t_1,\dots,t_{n-1})\, : \, t_i \geq 0,\sum_{i=1}^{n-1} t_i \leq 1  \} \]
The Jacobian of $\varphi$ is given by $|z_1 \wedge z_2 \wedge \dots \wedge z_{n-1}| = (n-1)!E_0(\Delta)/2$, see for instance \cite[Formula 9.12.4.3]{berger2009geometry}, so that
\begin{equation}\label{eq:gy:Ep Delta varphi}
  E_{\phi}(\Delta)=2\int_{\varphi(T)}\phi(|x|) \dd \Hau^{n-1}(x)=(n-1)!E_0(\Delta)\int_T \phi(|\varphi(t)|) \dd t_1\ldots \dd t_{n-1}
\end{equation}

We will estimate the integral from below by restricting the domain of integration to the subregion $C_T \subset T$ defined as follows
\[
C_T := \{(t_1,\dots,t_{n-1})\in T: \frac{3}{4}\leq t_1 \leq 1 \}.
\]
For $ (t_1,\dots, t_{n-1})\in C_T$ we have that $\sum_{i=2}^{n-1} t_i \leq 1-t_1 \leq \frac{1}{4}$,
and in particular
\[
|\varphi(t)|=\left|t_1 z_1 + \sum_{i=2}^{n-1} t_i z_i\right| \geq t_1|z_1|-\sum_{i=2}^{n-1}t_i|z_i| \geq \left(t_1 - \sum_{i=2}^{n-1} t_i \right) |z_1| \geq \frac{1}{2} |z_1|
\]

Thus we can estimate $E_\phi(\Delta)$ in \eqref{eq:gy:Ep Delta varphi} as
\begin{align*}
E_\phi(\Delta)  \ge (n-1)!E_0(\Delta)\int_{C_T} \phi \left(\left| \varphi(t) \right| \right) \dd t_1 \dots \dd t_{n-1}
\ge (n-1)!E_0(\Delta) |C_T| \phi \left( \frac{1}{2}|z_1| \right).
\end{align*}
Since $C_T$ is a homothetic copy of $T$ scaled by $1/4$ and shifted, more precisely $C_T=1/4\,T+(3/4,0,\ldots,0)$, we have $|C_T| = 4^{-(n-1)}|T| = (4^{n-1}(n-1)!)^{-1}$. Substituting this yields
\begin{align*}
\frac{E_\phi(\Delta)}{E_0(\Delta)} \ge \frac{1}{4^{n-1}}\phi\left(\frac{1}{2}|z_1|\right),
\end{align*}
which concludes the proof.
\end{proof}

We are now in a position to prove Lemma \ref{lm:bound_on_max_radius_increasing}.

\begin{customproof}{Lemma \ref{lm:bound_on_max_radius_increasing}}

    Let $\Omega$ be a convex body in $\R^n$ and let $y_1 \in \partial\Omega$ be a point at maximum distance from the origin, so that $|y_1|=R$. Let $V_1:= \Span \{y_1\}$. Define $y_2 \in \p \Omega$ as a point farthest from $V_1$. More precisely, consider
\[
b_2:= \max \{ \dist(V_1,z) \, : \, z \in \p \Omega \}>0,
\]
and $y_2$ is the point where the maximum is achieved. Set $V_2=\Span \{y_1,y_2\}$. Iteratively define
\[
b_j:=\max \{ \dist(V_{j-1},z) \, : \, z \in \p \Omega \}, \quad \text{for }j=2,\dots, n-1,\;\text{ and }\; b_1:=R
\]
and the corresponding points $y_j$ where such maximum is achieved.  Notice that $b_j>0$ since  $\mathcal{H}^{n-1}(\p \Omega) >0$.
Let us define for $j=1,\ldots,n-1$, the $j$-dimensional simplex $\Delta_j$ of vertices $0,y_1, \dots, y_j$
by
$$
  \Delta_j:=\left\{ \sum_{i=1}^{j}t_iy_i;\,t_i\geq 0\,, \sum_{i=1}^{j}t_i\leq 1\right\},\,\text{ and }\, \Delta:=\Delta_{n-1}.
$$
It is easy to see by induction on $j$ that $\Delta$ is contained in $\Omega$.

The $(n-1)$-dimensional area $\mathcal{H}^{n-1}(\Delta)$ of the simplex $\Delta$ is given by (see again \cite[Formula 9.12.4.3]{berger2009geometry} for instance)
\begin{equation}\label{eq:gy:area of simplex}
  \frac{E_0(\Delta)}{2}=\mathcal{H}^{n-1}(\Delta)=\frac{1}{(n-1)!}b_1  \cdots  b_{n-1}.
\end{equation}

Moreover, the vertices of $\Delta$ are $\{0,y_1,\dots,y_{n-1} \}$ thus, by calling $P$ the projection of $\Omega$ onto $V_{n-1}$, we infer that $\Delta \subset P(\Omega)$. 
Since $\Delta\subset P(\Omega)$ and by \eqref{eq:gy:cauchy surface general}, we deduce 
\[ E_0(\Delta)\le E_0(P(\Omega))\le E_0(\p \Omega). 
\]
 Thus we obtain  that $E_0(\Delta)\leq E_0(P(\Omega)) \leq E_0(\p \Omega)$, and thus
\begin{equation}
 \label{eq:b times R const}
  b_2 \cdot \cdots \cdot b_{n-1}\leq (n-1)!\, \frac{E_0(\partial\Omega)}{2 b_1}
  =(n-1)!\, \frac{E_0(\partial\Omega)}{2 R}.
\end{equation}

By definition of $R$ we have that, up to a rotation, $\Omega \subset [-R,R] \times \R^{n-1}$. Similarly by construction and a simple induction argument, it follows that $\Omega$ is contained, up to a rotation, in the hyperrectangle
\begin{equation}
  \label{eq:Omega in cylinder}
\Omega\subset C_{R,b}:=[-R,R] \times [-b_2,b_2] \times  [-b_3,b_3] \times \cdots \times  [-b_{n-2},b_{n-2}] \times [-b_{n-1},b_{n-1}]^2.
\end{equation}

For future reference we define 
\[
b:= \left( b_2 \cdots b_{n-1}  \right)^{\frac{1}{n-2}},
\]
and  \eqref{eq:b times R const} becomes
\begin{equation} \label{eq:b small with R}
   b^{n-2} \leq (n-1)! \frac{E_0(\partial\Omega)}{2R}.
\end{equation}

Since $\phi$  is increasing and $\phi(0)\geq0$ we can make use of the monotonicity property of Theorem \ref{thm:monotonicity} and  compare the energy of the simplex $\Delta$ with the energy of $\partial \Omega$. This leads, using Lemma \ref{lm:est_on_triangle}, to
\begin{align*}
    E_\phi(\p \Omega) \geq E_\phi (\Delta) \geq
    \frac{1}{4^{n-1}}\, E_0(\Delta) \, \phi \left( \frac{R}{2}\right) .
\end{align*}
Finally, we show that the measure of the simplex is comparable to the perimeter of $\Omega$. Indeed, since $R\geq b_2 \geq \dots \geq b_{n-2} \geq b_{n-1}$, as a consequence of \eqref{eq:Omega in cylinder} one has that the two biggest faces of $\p C_{R,b}$ are
$$
   [-R,R] \times [-b_2,b_2] \times \cdots  \times [-b_{n-1},b_{n-1}]\times\{\pm b_{n-1}\}
$$
and each of them have $\mathcal{H}^{n-1}$-measure $2^{n-1} R \cdot b_2 \cdots b_{n-1}$.
Thus one gets
\begin{align*}
    E_0(\partial \Omega) \leq E_0
    (\p C_{R,b}) \leq n 2^n R \cdot b_2 \cdots b_{n-1} = 2^{n-1}\, n! \,E_0(\Delta).
\end{align*}
Combining both estimates one obtains that 
\begin{align*}
\frac{ E_\phi(\p \Omega) }{E_0(\partial \Omega) } \geq \frac{ E_\phi(\Delta) }{2^{n-1} n! E_0(\Delta) }\geq \frac{1}{2^{3n-3}\, n!} \, \phi \left( \frac{R}{2}\right)
=\frac{1}{8^{n-1}\,n!}  \, \phi\left(\frac{R}{2}\right).
\end{align*}
\end{customproof}

 As an application of the above mentioned results we obtain the following theorem on the existence of minimizers. This theorem is not new and is a special case of \cite[Theorem $1.1$]{CGR}. However, the results of the present paper allow for a short and self-contained proof.  

\begin{Theorem}
\label{cor:intro:existence1}
Let $n\geq 2$, $\phi$ be an increasing weight function with $\phi(0)\geq 0$ and $\lim_{t\to \infty} \phi(t)=\infty$. Then, for every $\lambda>0$ there exists a minimizer  of 
\begin{equation} \label{min_problem_p_pos}
   E_\phi(\p \Omega)\quad\text{among all convex bodies $\Omega\subset\R^n$ with $E_0( \p \Omega)=\lambda$.}
\end{equation}
\end{Theorem}

\begin{proof}
Let $(\Omega_i)_{i\geq 1}$ be a minimizing sequence of convex bodies with  $E_0( \p \Omega_i)=\lambda$ for all $i\geq 1$.
We observe that we can assume $0 \in {\Omega_i}$ for all $i \ge 1$ without loss of generality, since translating the $\Omega_i$ towards the origin decreases the Euclidean norm of all its points, thanks to the convexity of $\Omega_i$. This operation decreases the energy $E_\phi$.

 We claim that the sequence of sets $(\Omega_i)_{i\geq 1}$ is bounded; that is, there exists $R>0$ such that $\Omega_i\subset B_R$ for all $i\geq 1$. 
We argue by contradiction and assume the sequence is unbounded. 
Since we established that $0 \in {\Omega_i}$ for all $i\geq1$, we must have that
\begin{equation}\label{eq:gy:diam infty}
\lim_{i\to\infty}\text{diam}(\Omega_i)=\infty.
\end{equation}

If $n=2$, this immediately yields a contradiction, as the diameter of a convex set is universally bounded by its perimeter, which is equal to  $\lambda$ for all $i$. On the other hand, for $n\geq 3$, we will deduce that \eqref{eq:gy:diam infty} cannot hold by Lemma \ref{lm:bound_on_max_radius_increasing}, which is applicable since $0 \in {\Omega_i}$. By defining $R_i := \max_{x \in \partial\Omega_i} |x|$, the unbounded diameter assumption implies $R_i \to \infty$. Thus Lemma \ref{lm:bound_on_max_radius_increasing} together with the assumption $\lim_{t\to\infty}\phi(t)=\infty$ yields that $E_\phi(\p \Omega_i) \to \infty$. This contradicts the assumption that the sequence is minimizing.
This shows that \eqref{eq:gy:diam infty} cannot hold and that the sequence is bounded. 

From the Blaschke selection theorem, see for instance \cite[Theorem 1.8.7]{Schneider}, we infer that, up to a subsequence, $\Omega_i$ converge in the Hausdorff metric to a limiting compact convex set $\Omega_\infty$. The existence of the minimizer is then readily established by combining the continuity of the perimeter and of the energy $E_\phi$, which are guaranteed thanks to Theorem \ref{thm:continuity_energy}.
\end{proof}


\medskip
\noindent{\bf Data Availability} Not applicable. 

\smallskip
\noindent{\bf Conflict of interest} On behalf of all authors, the corresponding author states that there is no conflict of interest.

	\bibliographystyle{abbrv}
	\bibliography{Bib}
\end{document}